\documentclass[a4paper,12pt,reqno]{amsart}
\usepackage{latexsym}
\usepackage{amssymb,amsfonts,amsmath,mathrsfs}
\usepackage{amsmath, amssymb, amsthm, epsfig}
\usepackage{hyperref, latexsym}
\usepackage{url}

\usepackage{color}
 \newtheorem{theorem}{Theorem}
 \newtheorem{lemma}[theorem]{Lemma}

 \theoremstyle{definition}

 \theoremstyle{remark}

 \newcommand{\C}{\mathbb{C}}
 \newcommand{\R}{\mathbb{R}}
 \newcommand{\N}{\mathbb{N}}

 \newcommand{\hh}{\tfrac12}

 \newcommand{\dt}{\text{\rm d}t}
 \newcommand{\du}{\text{\rm d}u}

 \newcommand{\dw}{\text{\rm d}w}
 \newcommand{\dx}{\text{\rm d}x}
 \newcommand{\dk}{\text{\rm d}k}
 \newcommand{\dy}{\text{\rm d}y}
 \newcommand{\dl}{\text{\rm d}\lambda}

\newcommand{\dsigma}{\text{\rm d}\sigma}
\newcommand{\dalpha}{\text{\rm d}\alpha}

 \newcommand{\dmu}{\text{\rm d}\mu}
 \newcommand{\dnu}{\text{\rm d}\nu}

\newcommand{\Imp}{\textrm{Im}}

 \newcommand{\new}{}
  \newcommand{\sgn}{sgn}

\author[Carneiro, Chandee, and Milinovich]{Emanuel Carneiro, Vorrapan Chandee, and Micah B. Milinovich}
\address{ Instituto de Matematica Pura e Aplicada (IMPA), Estrada Dona Castorina, 110, Rio de Janeiro, RJ, 22460-320, Brazil}
\email{carneiro@impa.br}
\thanks{EC supported by CNPq grants 473152/2011-8 and 302809/2011-2. VC is supported by CRM-ISM fellowship. MBM is supported by an AMS-Simons Travel Grant and an NSA Young Investigator grant.}
\address{Department of Mathematics, Burapha University, 169 Long-Hard Bangsaen Road, Saen Sook Sub-district, Mueang District, Chonburi, 20131, Thailand}
\address{Centre de recherches math\'ematiques, Universit\'e de Montr\'eal, P.O. Box 6128, Centre-ville Station, Montreal, QC, H3C 3J7, Canada}
\email{vorrapan@gmail.com}

\address{Department of Mathematics, University of Mississippi, University, MS 38677 USA}
\email{mbmilino@olemiss.edu}


\title[Bounding $S(t)$ and $S_1(t)$]{Bounding $S(t)$ and $S_1(t)$ on the Riemann hypothesis}

\begin{document}

\maketitle

\centerline{\it To Professor Steven M. Gonek on the occasion of his sixtieth birthday.}

\vskip 0.2 in

\begin{abstract}
Let $\pi S(t)$ denote the argument of the Riemann zeta-function, $\zeta(s)$, at the point $s=\frac{1}{2}+it$. Assuming the Riemann hypothesis, we present two proofs of the bound
\begin{equation*}
|S(t)| \leq \left( \tfrac{1}{4} + o(1) \right)\tfrac{\log t}{\log \log t}
\end{equation*}
 for large $t$. This improves a result of Goldston and Gonek by a factor of 2. The first method \new{consists in bounding} the auxiliary function $S_1(t) = \int_0^{t} S( u) \> \du$ using extremal functions constructed by Carneiro, Littmann and Vaaler. We then relate the size of $S(t)$ to the size of the functions $S_1(t\pm h)-S_1(t)$ when $h\asymp 1/\log\log t$. The alternative approach bounds $S(t)$ directly, relying on the solution of the Beurling-Selberg extremal problem for the odd function $f(x) = \arctan\left(\tfrac{1}{x}\right) - \tfrac{x}{1 + x^2}$. This draws upon recent work \new{by} Carneiro and Littmann.
\end{abstract}

\section{Introduction}
Let $\zeta(s)$ denote the Riemann zeta-function and, as usual, let $N(t)$ denote the number of zeros $\rho=\beta+i\gamma$ of $\zeta(s)$ with ordinates $\gamma$ in the interval $(0,t]$. Then, for $t\geq 2$,
\begin{equation}\label{N(T)}
N(t) = \frac{t}{2\pi} \log \frac{t}{2\pi} -\frac{t}{2\pi} + \frac{7}{8} + S(t) + O\Big(\frac{1}{t}\Big), 
\end{equation}
where, if $t$ is not an ordinate of a zero of $\zeta(s)$, $S(t)$ denotes the value of $\frac{1}{\pi}\arg \zeta\big(\frac{1}{2}+it\big)$ obtained by continuous variation along the   line segments joining the points $2, 2 + it$, and $\frac{1}{2} + it$, taking the argument of  $\zeta(s)$ at $s=2$ to be zero. If $t$ is an ordinate of a zero of $\zeta(s)$ we set
\begin{equation*}\label{S(gamma)}
 S(t) = \tfrac{1}{2} \ \! \lim_{\varepsilon\to 0} \big\{ S(t\!+\! \varepsilon)\!+\!S(t\!-\! \varepsilon) \big\}.
 \end{equation*}
\smallskip 

Assuming the Riemann hypothesis (RH), Littlewood \cite{L} proved that
\begin{equation}\label{Littlewood1}
|S(t)| \leq \big(C + o(1)\big) \frac{\log t}{\log \log t}.
\end{equation}
Here $C$ is a constant and $o(1)$ denotes a quantity which tends to $0$ as $t$ grows. The order of magnitude of this estimate has never been improved, and advances have rather focused on diminishing the value of the admissible constant $C$. In \cite{RS} Ramachandra and Sankaranarayanan showed that $C=1.119$ is admissible in \eqref{Littlewood1}, and later Fujii \cite{F} obtained the result with $C=0.67$. Recently, the theory of extremal functions of exponential type has proved useful in this context. Goldston and Gonek \cite{GG}, exploring the relation between the functions $S(t)$ and $N(t)$, used the classical Beurling-Selberg majorants and minorants of characteristic functions of intervals, together with the Guinand-Weil explicit formula for the zeros of $\zeta(s)$, to obtain the bound
\begin{equation}\label{GoGo}
|S(t)| \leq \left(\frac{1}{2} + o(1)\right) \frac{\log t}{\log \log t}.
\end{equation}
\smallskip

Following Goldston and Gonek's work, Chandee and Soundararajan \cite{CS} recognized that similar techniques could be used to estimate the size of $\big|\zeta\big(\frac{1}{2}+it\big)\big|.$ To obtain their bound, they made use of the extremal functions for $f(x) = \log\big(\frac{4+x^2}{x^2}\big)$, available in the framework of Carneiro and Vaaler \cite{CV2}. Their method likely represents the limit of existing methods for bounding $\big|\zeta\big(\frac{1}{2}+it\big)\big|$ assuming RH. Later, Carneiro and Chandee \cite{CC} extended this result to bounding $|\zeta(s)|$ for $s$ in the critical strip, by using the extremal majorants and minorants for the family of functions  $f_{\alpha}(x) = \log\big(\frac{4+x^2}{(\alpha - 1/2)^2 + x^2}\big)$. This extremal problem is solvable in the the more general setting of Gaussian subordination for even functions using the work of Carneiro, Littmann, and Vaaler \cite{CLV}.
\smallskip

Inspired by these previous results, we use similar techniques to bound the auxiliary function
$$ S_1(t) := \int_0^t S(u) \ \! \du.$$
There has been some earlier work on establishing explicit bounds for $S_1(t)$. Littlewood \cite{L} was the first to prove that $ S_1(t)  \ll  \log t/(\log\log t)^2$ under the assumption of the Riemann hypothesis.
More recently, also assuming RH, Karatsuba and Korol\"{e}v \cite{KK} showed that 
\begin{equation*}
 \big| S_1(t)\big| \ \! \leq \ \! \big(40 \! +\!o(1)\big) \frac{ \log t}{ (\log\log t)^2},\\
\end{equation*}
and Fujii \cite{F} obtained the bounds
\begin{equation*}
 - \big(0.51 \!+\!o(1)\big) \frac{ \log t}{ (\log\log t)^2} \ \! \leq \ \! S_1(t) \ \! \leq \ \! \big(0.32 \! +\!o(1)\big) \frac{ \log t}{ (\log\log t)^2}.\\
\end{equation*}
In this paper, we derive upper and lower bounds for the function $ S_1(t)$ using the explicit formula and the theory of extremal functions of exponential type for the function  $f_1(x) = 1-x\arctan(1/x).$ Recent work of Carneiro, Littmann, and Vaaler \cite{CLV} allows us to choose these extremal functions in an optimal way and derive the following theorem.

\begin{theorem}\label{th2}
Assume RH. For $t$ sufficiently large we have
\begin{equation*}
 - \left(\frac{\pi}{24} \!+\!o(1)\right) \frac{ \log t}{ (\log\log t)^2} \ \! \leq \ \! S_1(t) \ \! \leq \ \! \left(\frac{\pi}{48} \! +\!o(1)\right) \frac{ \log t}{ (\log\log t)^2}\,,\\
\end{equation*} 
where the terms $o(1)$ in the above inequalities are $O(\log\log\log t/\log\log t)$.
\end{theorem}
 
As a consequence of the above theorem, we are able to improve Goldston and Gonek's bound for $S(t)$ by a factor of 2.

\begin{theorem}\label{th1}
Assume RH. For $t$ sufficiently large we have
\begin{equation*}
|S(t)| \leq \frac{1}{4} \frac{\log t}{\log \log t} + O\left(\frac{\log t \log \log \log t}{(\log \log t)^2}\right).
\end{equation*}
\end{theorem}
\smallskip

Previously, it had been suggested that the estimate in (\ref{GoGo}) had attained the limit  of existing methods for bounding $S(t).$ However, the result in the above theorem is actually not too surprising since Goldston and Gonek derived their inequality for $S(t)$ from the bound 
\begin{equation} \label{StplushMinusSt}
\displaystyle{|S(t\!+\!h)-S(t)| \leq \left(\frac{1}{2}\!+\!o(1)\right) \frac{\log t}{\log\log t}},
\end{equation}
where $0 < h \leq \sqrt{t},$ instead of by bounding $S(t)$ directly. Given the range of uniformity of this estimate, it is reasonable to believe that there is a $t$ and an $h$ such that simultaneously $S(t)$ is large and positive, and $S(t + h)$ is large and negative (or vice versa) for some $h$ satisfying $0<h\leq \sqrt{t}$. Hence, the estimate for $|S(t+h)-S(t)|$  in (\ref{StplushMinusSt}) suggests the pointwise bound $|S(t)|\leq (\frac{1}{4}+o(1))\log t/\log\log t$ provided by Theorem \ref{th1}. Moreover, as a consequence of Theorem \ref{th1}, we are able to rederive  (but not improve) the bound in (\ref{StplushMinusSt}). This is an important observation since Goldston and Gonek showed how to use (\ref{StplushMinusSt}) to bound the maximum multiplicity of a zero of $\zeta(s)$ and how to bound the maximum gap between consecutive zeros of the Riemann zeta-function on the critical line (see \cite[Corollary 1]{GG}). Hence, Theorem \ref{th1} can be used to recover these results about the zeros of $\zeta(s)$, but not to improve them.
\smallskip

In Section 4 of this paper, we will show how to obtain the bound for $S(t)$ in Theorem \ref{th1}  by relating the size of $S(t)$ to the size of the functions $ S_1(t\pm h)- S_1(t)$ for $h \asymp 1/\log\log t$ and then invoking Theorem \ref{th2}. This is reminiscent of the work of Selberg \cite{Selberg}, Ghosh \cite{Ghosh}, and Tsang \cite{Tsang}  who used the behavior of the function $ S_1(t+h)- S_1(t)$ to prove omega-theorems for $S(t)$ and to study the sign changes of $S(t)$.
\smallskip

Although the method described above can be used to improve existing bounds for $S(t)$, this approach does not bound $S(t)$ directly.  Rather, it derives the bound for $S(t)$ from a corresponding bound on $S_1(t)$ which in turn arises from an application of the theory of extremal functions to $f_1(x) = 1-x\arctan(1/x).$  A natural question to ask is whether the theory of extremal functions of exponential type could be applied directly to bound $S(t).$ It turns out that a direct approach is possible and, in Section \ref{Sec:StviaExtremal}, we show how to obtain the bound for $S(t)$ in Theorem \ref{th1} via extremal functions.  However, in constrast with previous work on optimal bounds in the theory of $\zeta(s)$, this approach to bound $S(t)$ is connected with the Beurling-Selberg extremal problem for an {\it odd function}, namely
\begin{equation*}
f(x) = \arctan\left(\frac{1}{x}\right) - \frac{x}{1 + x^2}.
\end{equation*}
It is shown in Section \ref{Sec:StviaExtremal} that the solution of this particular problem is provided by the recent Gaussian subordination framework for truncated (and odd) functions of Carneiro and Littmann \cite{CL}.
\smallskip

The bound for $S(t)$ via the extremal function approach is the same as the one deduced from the estimate of $ S_1(t\pm h)- S_1(t)$. This should be expected since these two parallel approaches use essentially the same methods, and the same assumptions.  The main assumption is that all non-trivial zeros of $\zeta(s)$ lie on the critical line, since the method depends on an explicit formula which relates information on the zeta-function to the zeros of $\zeta(s)$ and the prime numbers.  The theory of extremal functions is used to majorize certain naturally occuring test functions. Note that we use the the function $f(x) = \arctan\left(\frac{1}{x}\right) - \frac{x}{1 + x^2} =- \frac{d}{dx} f_1(x)$ to bound $S(t)$ where $f_1(x) = 1 - x\arctan\left(\frac{1}{x}\right)$ is the function used to bound $S_1(t)$. Since $\frac{d}{dt}S_1(t)=S(t)$ almost everywhere, it is not surprising that the two methods lead to the same result (namely Theorem \ref{th1}).
\smallskip

Although the inequality in Theorem \ref{th1} appears to be the best known bound for $S(t)$ assuming the Riemann hypothesis, the true size of $S(t)$ is perhaps much smaller.  For instance, a heuristic argument of Farmer, Gonek and Hughes \cite{FGH} suggests that $S(t) = O(\sqrt{\log t \log\log t}).$ This result seems unattainable using the present method without a significant breakthrough in estimating certain prime number sums that arise when applying the Guinand-Weil explicit formula. In the present article, these sums are estimated trivially.


\section{Preliminary results}
Before proving Theorem \ref{th2} and Theorem \ref{th1}, we collect some preliminary results. As mentioned in the introduction, our bounds for $ S_1(t)$ require extremal majorants and minorants for the function $f_1(x) = 1 - x\arctan(1/x) $. The connection between $S_1(t)$ and $f_1(x)$ is made explicit in the following lemma. 

\begin{lemma}\label{lemmafSt} Assume RH.  Let 
\begin{equation*} 
f_1(x) = 1 - x\arctan\left(\frac{1}{x}\right).
\end{equation*}
Then, for $t\geq 2$, we have
\begin{equation} \label{eqn:fForSt}
S_1(t) = \frac{1}{4\pi} \log t - \frac{1}{\pi} \sum_{\gamma} f_1(t\!-\!\gamma) + O(1),
\end{equation}
where the sum runs over the non-trivial zeros $\rho = \frac{1}{2}+i\gamma$ of $\zeta(s)$.
\end{lemma}

\begin{proof}
By \cite[Theorem 9.9] {T} we have
\begin{equation}\label{eqn:mainS2}
 S_1(t) = \frac{1}{\pi} \int_{1/2}^{3/2}  \log\big|\zeta(\sigma+it)\big| \ \! \dsigma + O(1).
 \end{equation}
In order to prove the lemma, we replace the integrand by an absolutely convergent sum over the zeros of $\zeta(s)$ and then integrate term-by-term.
\smallskip

Let $\xi(s) = \frac{1}{2}s(s-1)\pi^{-s/2}\Gamma(\frac{s}{2})\zeta(s)$ denote Riemann's $\xi$-function. Then $\xi(s)$ is an entire function of order 1 and the zeros of $\xi(s)$ correspond to the non-trivial zeros of $\zeta(s)$. By Hadamard's factorization formula (cf. \cite[Chapter 12]{Dav}), we have
$$ \xi(s) = e^{A+Bs} \prod_\rho \left(1-\frac{s}{\rho} \right) e^{s/\rho},$$
where $\rho$ runs over the non-trivial zeros of $\zeta(s)$, $A$ is a constant and $B = -\sum_\rho \text{Re}(1/\rho)$. Note that Re($1/\rho)$ is positive and that $\sum_\rho \text{Re}(1/\rho)$ converges absolutely. Assuming the Riemann hypothesis, it follows that
$$ \left|\frac{\xi(\sigma+it)}{\xi\big(\frac{3}{2}+it\big)}\right| = \prod_\rho \left( \frac{\big(\sigma\!-\!\frac{1}{2}\big)^2+(t\!-\!\gamma)^2}{1+(t\!-\!\gamma)^2} \right)^{\frac{1}{2}}.$$
Hence
$$ \log\big|\xi(\sigma+it)\big| - \log\big|\xi\big(\tfrac{3}{2}+it\big)\big| = \frac{1}{2} \sum_\rho \log\left(  \frac{\big(\sigma\!-\!\frac{1}{2}\big)^2+(t\!-\!\gamma)^2}{1+(t\!-\!\gamma)^2}  \right).$$
By Stirling's formula for $\Gamma(s)$, we find that
\begin{equation}\label{eqn:mainS3}
 \log\big|\zeta(\sigma+it)\big| = \left( \tfrac{3}{4}-\tfrac{\sigma}{2} \right) \log t - \frac{1}{2} \sum_\rho \log\left(  \frac{1+(t\!-\!\gamma)^2} {\big(\sigma\!-\!\frac{1}{2}\big)^2+(t\!-\!\gamma)^2} \right) +O(1)
  \end{equation}
 uniformly for $\frac{1}{2}\leq \sigma \leq \frac{3}{2}$ and $t\ge 2$, say. Moreover, for $\frac{1}{2}\leq \sigma \leq \frac{3}{2}$ and $t\ge 1$, if $\sigma +it$ is not a zero of $\zeta(s)$ then (\ref{N(T)}) implies that  the sum over zeros on the right-hand side of \eqref{eqn:mainS3} converges absolutely (one can also deduce this from the fact that the product in Hadamard's factorization converges absolutely to a non-zero value). In case when $\sigma +it$ corresponds to a zero of $\zeta(s)$, both sides of \eqref{eqn:mainS3} are $-\infty$. Inserting (\ref{eqn:mainS3}) into (\ref{eqn:mainS2}), we deduce that
\begin{equation*}\label{eqn:mainS1}
\begin{split}
 S_1(t) &= \frac{1}{\pi} \int_{1/2}^{3/2} \left( \tfrac{3}{4}-\tfrac{\sigma}{2} \right) \log t  \ \! \dsigma  
 \\
 & \quad\quad\quad - \frac{1}{2\pi} \int_{1/2}^{3/2} \sum_\rho \log\left(  \frac{1+(t\!-\!\gamma)^2} {\big(\sigma\!-\!\frac{1}{2}\big)^2+(t\!-\!\gamma)^2} \right) \ \! \dsigma + O(1)
 \\
 &= \frac{1}{4\pi} \log t - \frac{1}{\pi} \sum_{\rho} f_1(t-\gamma) + O(1),
 \end{split}
 \end{equation*}
where the function $f_1(x)$ is defined by
\begin{equation*} \label{eqn:f_interm_integralLog}
 f_1(x) = \frac{1}{2} \int_{1/2}^{3/2} \log\left(  \frac{1+x^2}{\big(\sigma\!-\!\frac{1}{2}\big)^2+x^2} \right) \dsigma = 1 - x\arctan\left(\frac{1}{x}\right),
\end{equation*}
and the interchange of the integral and the sum is justified by monotone convergence since the terms involved are non-negative. This completes the proof of the lemma.
\end{proof}

\noindent{\it Remark.} In the proof of the above lemma, the first step is to divide $|\xi(\sigma+it)|$ by $|\xi(\tfrac{3}{2}+it)|$. This simplifies matters when applying Stirling's formula and ensures that the resulting sum over zeros of $\zeta(s)$ is absolutely convergent and contains only non-negative terms, which justifies the interchange of the integrals. A similar trick is necessary in the proof of Lemma \ref{S(t)sum}, below. In particular, we need to subtract a factor of $\zeta'(3/2+it)/\zeta(3/2+it)$ in a certain integral representation of $S(t)$ in order to ensure the absolute convergence of the resulting sum over zeros of $\zeta(s)$. 
\bigskip

In Section \ref{Sec:StviaExtremal}, we bound $S(t)$ directly using extremal functions. This requires an expression relating $S(t)$ to the function $f(x)=\arctan\left(\frac{1}{x}\right)-\frac{x}{1+x^2}$ mentioned in the introduction. Our next lemma expresses $S(t)$ as a sum of $f(x)$ over the imaginary parts of the zeros of the zeta-function plus a bounded error term.

\begin{lemma} \label{S(t)sum}
Assume RH. Let 
$$\displaystyle{f(x) = \arctan\left(\frac{1}{x}\right) - \frac{x}{1 + x^2} }.$$ 
Then, for $t\geq 2$ and $t$ not coinciding with an ordinate of a zero of $\zeta(s)$, we have
\begin{equation}\label{Sec2.6.0}
S(t) = \frac{1}{\pi}  \sum_{\gamma} f(t-\gamma) + O(1),
\end{equation}
where the sum runs over the non-trivial zeros $\rho=\frac{1}{2}+i\gamma$ of $\zeta(s)$.
\end{lemma}

\begin{proof} For $t$ not coinciding with an ordinate of a zero of $\zeta(s)$, we have 
$$S(t) = - \frac{1}{\pi} \int_{\frac{1}{2}}^\infty \Imp \frac{\zeta'}{\zeta} (\sigma + it) \,\dsigma = \frac{1}{\pi}  \int_{\frac{3}{2}}^{\frac12}  \Imp \frac{\zeta'}{\zeta} (\sigma + it) \,\dsigma + O(1).$$
We now replace the integrand on the right-hand side of the expression above by an absolutely convergent sum over the non-trivial zeros of $\zeta(s)$.
\smallskip

Let $s=\sigma+it$. If $s$ is not a zero of $\zeta(s)$, then the partial fraction decomposition for $\zeta'(s)/\zeta(s)$ (cf. \cite[Chapter 12]{Dav}) and Stirling's formula for $\Gamma'(s)/\Gamma(s)$ imply that
\begin{equation} \label{partialfrac}
\begin{split}
 \frac{\zeta'}{\zeta}(s) &= \sum_\rho \left( \frac{1}{s-\rho}+\frac{1}{\rho} \right) - \frac{1}{2}\frac{\Gamma'}{\Gamma}\left(\frac{s}{2}+1\right) + O(1)
 \\
 &=  \sum_\rho \left( \frac{1}{s-\rho}+\frac{1}{\rho} \right) - \frac{1}{2} \log \left(\frac{t}{2}\right) + O(1)
 \end{split}
 \end{equation}
uniformly for $\frac{1}{2}\leq\sigma \leq \frac{3}{2}$ and $t\geq 2$, where the sum runs over the non-trivial zeros $\rho$ of $\zeta(s)$.  
Now suppose that $t$ is not the ordinate of a zero of $\zeta(s)$. Then, from (\ref{partialfrac}) and the Riemann hypothesis, it follows that
\begin{align*}\label{Sec2.4}
\begin{split}
S(t) = \frac{1}{\pi} & \int_{\frac{3}{2}}^{\frac12}  \Imp \frac{\zeta'}{\zeta} (\sigma + it) \,\dsigma + O(1)\\
& =  \frac{1}{\pi} \int_{\frac32}^{\frac12} \Imp\left( \frac{\zeta'}{\zeta} (\sigma + it) - \frac{\zeta'}{\zeta} \big(\tfrac{3}{2} + it\big)\right)  \dsigma + O(1)\\
& \ \ \ \ \ \ = \frac{1}{\pi} \int_{\frac12}^{\frac32} \sum_{\gamma}\left\{ \frac{(t - \gamma)}{(\sigma - \hh)^2 + (t - \gamma)^2} - \frac{(t - \gamma)}{1 + (t - \gamma)^2}\right\} \dsigma+ O(1)\\
& \ \ \ \ \ \ \ \ \ \ \ = \frac{1}{\pi}\sum_{\gamma} \int_{\frac12}^{\frac32} \left\{ \frac{(t - \gamma)}{(\sigma - \hh)^2 + (t - \gamma)^2} - \frac{(t - \gamma)}{1 + (t - \gamma)^2}\right\} \dsigma+ O(1)\\
& \ \ \ \ \ \ \ \ \ \ \ \ \ \ \ \ \ = \frac{1}{\pi} \sum_{\gamma} \left\{ \arctan \left( \frac{1}{t-\gamma}\right) - \frac{(t-\gamma)}{1 + (t - \gamma)^2}\right\} + O(1)
\\
& \ \ \ \ \ \ \ \ \ \ \ \ \ \ \ \ \ \ \ \ \ = \frac{1}{\pi}  \sum_{\gamma} f(t-\gamma) + O(1),
\end{split}
\end{align*}
where the interchange of the integral and the sum is justified by dominated convergence since $f(x)=O(x^{-3})$. 
This proves the lemma.
\end{proof}

The main idea in the proofs of Theorem \ref{th2} and Theorem \ref{th1} 
is to bound $S_1(t)$ and $S(t)$ from above and below by replacing the function $f_1(x)$ in \eqref{eqn:fForSt} and $f(x)$ in \eqref{Sec2.6.0} by an appropriate majorant or minorant of exponential type (thus with a compactly supported Fourier transform by the Paley-Wiener theorem). We then apply the following version of the Guinand-Weil explicit formula which connects these sums over the zeros of the zeta-function to sums of the Fourier transforms evaluated at the prime powers. 

\begin{lemma}
\label{lem:explicitformula} Assume RH.
Let $h(s)$ be analytic in the strip $|\text{\rm Im }s| \leq 1/2 + \varepsilon$ for some $\varepsilon > 0$, 
and assume that $|h(s)| \ll (1+|s|)^{-(1+\delta)}$ for 
some $\delta >0$ when $|\text{\rm Re }s|\to \infty$.  Let $h(w)$ be a real-valued for real $w$, and set $\widehat{h}(x) = \int_{-\infty}^{\infty} h(w) e^{-2\pi i x w} \> \dw$. 
Then 
\begin{align*}
\sum_{\rho} h(\gamma) &= h\left(\frac{1}{2i} \right) +h\left(-\frac{1}{2i}\right) 
- \frac{1}{2\pi} {\widehat h}(0) \log \pi  + \frac{1}{2 \pi}  \int_{-\infty}^{\infty} h(u) \ \! \text{\rm Re }\frac{\Gamma '}{\Gamma}\left(\frac{1}{4}   + \frac{iu}{2} \right) \> \du \\
& \ \ \ \ \ \ \ \ \ \ \ \ \ - \frac{1}{2\pi}\sum_{n=2}^{\infty} \frac{\Lambda(n)}{\sqrt{n}}\left(\widehat{h}\left( \frac{\log n}{2\pi }\right) + \widehat{h}\left( \frac{-\log n}{2\pi }\right)\right),
\end{align*}
where $\Gamma'/\Gamma$ is the logarithmic derivative of the gamma function, and $\Lambda(n)$ is the von Mangoldt function defined to be $\log p$, if $n=p^m$, $p$ a prime and $m\geq 1$ an integer, and zero otherwise.
\end{lemma} 

\begin{proof}  The proof of this lemma follows from \cite[Theorem 5.12]{IK}. It can be stated unconditionally by replacing $h(\gamma)$ with $h((\rho{-}1/2)/i)$ in the sum over non-trivial zeros on the left-hand side of the above identity.
\end{proof}


\section{Bounding $S_1(t)$: Proof of Theorem \ref{th2}}
As mentioned in the previous section, the proof of Theorem \ref{th2} requires appropriate majorant and minorant functions of exponential type for the function  $f_1(x) = 1 - x\arctan(1/x).$  The required properties of these extremal functions are described in the following lemma.  

\begin{lemma} \label{lem:minorant} 
Let $f_1(x) = 1 - x\arctan(1/x)$ and let $\Delta\ge 1$. Then there are unique real entire functions $g_{\Delta}^-: \mathbb{C} \rightarrow \mathbb{C}$ \ and \  $g_{\Delta}^+ : \mathbb{C} \rightarrow \mathbb{C}$ \ satisfying the following properties: 
\begin{enumerate}
\item[(i)] For real $x$ we have 
 \begin{equation}\label{ineq:boundforgDeltaxf1}
 \frac{-C }{1+x^2} \le g_{\Delta}^-(x) \le f_1(x) \le g_{\Delta}^+(x) \le \frac{C}{1+x^2}, 
\end{equation}
for some positive constant $C$.  Moreover, for any complex number 
 $z=x+iy$ we have 
 \begin{equation} \label{ineq:boundforgDeltazf1}
 \big|g_{\Delta}^\pm(z)\big| \ll  \frac{\Delta^2  }{(1\!+\! \Delta |z|)}e^{2\pi \Delta|y|}. \\
 \end{equation}
 
\item[(ii)]  The Fourier transforms of $g_{\Delta}^\pm$, namely 
 $$
 {\widehat g}_{\Delta}^\pm(\xi) = \int_{-\infty}^{\infty} g_{\Delta}^\pm(x) e^{-2\pi i x \xi } \ \! \dx,  
 $$ 
 are continuous functions supported on the interval $[-\Delta, \Delta]$ and satisfy 
 \begin{equation*}
 \big|{\widehat g}_{\Delta}^\pm(\xi)\big| \ll 1
 \end{equation*}
 for all $\xi \in [-\Delta, \Delta]$, where the implied constant is independent of $\Delta$. 
\smallskip
 
 \item[(iii)] The $L^1$-distances of $g_\Delta^\pm$ to $f_1$ are given by
\begin{equation*} \label{integral:lower}
 \int_{-\infty}^{\infty}  \big\{ f_1(x) - g_{\Delta}^-(x)\big\} \ \! \dx =  \int_{1/2}^{3/2} \frac{1}{\Delta} \Big\{\log\big(1 + e^{-2\pi \Delta (\sigma - 1/2)}\big)- \log \big(1 + e^{-2\pi\Delta}\big)\Big\} \ \! \dsigma
 \end{equation*}
and
\begin{equation*} \label{integral:upper}
\int_{-\infty}^{\infty}  \big\{g_{\Delta}^+(x) - f_1(x)\big\} \ \! \dx = -\int_{1/2}^{3/2} \frac{1}{\Delta} \Big\{\log\big(1 - e^{-2\pi \Delta (\sigma - 1/2)}\big)- \log \big(1 - e^{-2\pi\Delta}\big)\Big\} \ \! \dsigma.
\end{equation*}
 \end{enumerate}
 \end{lemma}

\smallskip

We postpone the proof of this lemma until Section \ref{Sec:Extremal}, and proceed with the proof of Theorem \ref{th2}.

\begin{proof}[Proof of Theorem \ref{th2}]
Let $h^\pm_\Delta(z)=g_\Delta^\pm(t-z)$ so that $\widehat{h}^\pm_\Delta(\xi)= \widehat{g}^\pm_\Delta(-\xi) e^{-2\pi i \xi t}.$ By Lemma \ref{lemmafSt} and (i) of Lemma \ref{lem:minorant}, we have for any positive $\Delta$ that
\begin{equation}\label{em0}
\frac{1}{4\pi} \log t - \frac{1}{\pi} \sum_\gamma h^+_\Delta(\gamma) + O(1)\leq S_1(t) \leq \frac{1}{4\pi} \log t - \frac{1}{\pi}\sum_\gamma h^-_\Delta(\gamma) + O(1).
\end{equation}
From (i) and (ii) of Lemma \ref{lem:minorant}, we find that 
\begin{equation}\label{em1}
\left|\widehat{h}^\pm_\Delta(0)\right| \ll 1,
\end{equation}
\begin{equation}\label{em2}
\left|h^\pm_\Delta\left(\pm\frac{1}{2i}\right) \right|\ll \frac{\Delta^2 e^{\pi \Delta}}{1+\Delta t },
\end{equation}
and that
\begin{equation}\label{em3}
\sum_{n=2}^{\infty} \frac{\Lambda(n)}{\sqrt{n}}\left(\widehat{h}^\pm_\Delta\left( \frac{\log n}{2\pi }\right) + \widehat{h}^\pm_\Delta\left( \frac{-\log n}{2\pi }\right)\right) \ll \sum_{n\leq e^{2\pi\Delta}} \frac{\Lambda(n)}{\sqrt{n}} \ll e^{\pi \Delta},
\end{equation}
where the last inequality follows from the Prime Number Theorem and summation by parts. Inserting estimates \eqref{em1}, \eqref{em2}, and \eqref{em3} into Lemma \ref{lem:explicitformula} gives us
\begin{equation}\label{em4}
\sum_\gamma h^\pm_\Delta(\gamma) = \frac{1}{2 \pi}  \int_{-\infty}^{\infty} h^\pm_\Delta(u) \ \! \text{\rm Re }\frac{\Gamma '}{\Gamma}\left(\frac{1}{4}   + \frac{iu}{2} \right) \du + O\left(e^{\pi \Delta} +1 \right) + O\left( \frac{\Delta^2 e^{\pi \Delta}}{1+\Delta t} \right).
\end{equation}
Using Lemma \ref{lem:minorant}, we estimate the integral on the right-hand side of the above expression separately for each of the functions $h^+_\Delta$ and $h^-_\Delta$.
\smallskip

Using Stirling's formula for $\Gamma'(s)/\Gamma(s)$, parts (i) and (iii) of Lemma \ref{lem:minorant}, and the identity
 \begin{equation}\label{pi}
 \int_{-\infty}^{\infty}f_1(x) \ \! \dx = \frac{\pi}{2},
 \end{equation}
it follows that
\begin{align} \label{useExplicit:Lower}
\begin{split} 
\frac{1}{2\pi }& \int_{-\infty}^{\infty} h^-_\Delta(u) \,\text{Re } \frac{\Gamma^{\prime}}{\Gamma} \left( 
 \frac{1}{4} + \frac{iu}{2} \right) \du 
 \\
 &= \frac{1}{2\pi} \int_{-\infty}^{\infty} g_\Delta^-(u) \Big(\log t + O\big(\log (2\!+\!|u|)\big)\Big) \> \du 
 \\
&= \frac{1}{2\pi} \int_{-\infty}^{\infty} \Big\{ f_1(u) - \big(f_1(u) -g_\Delta^-(u)\big)\Big\} \Big(\log t + O\big(\log (2\!+\!|u|)\big)\Big) \> \du 
\\
 &= \frac{1}{4} \log t - \frac{\log t}{2\pi \Delta} \int_{1/2}^{3/2} \left\{\log\big(1 + e^{-2\pi \Delta (\sigma - 1/2)}\big)- \log \big(1 + e^{-2\pi\Delta}\big)\right\} \dsigma + O(1)
 \\
 &\geq \frac{1}{4} \log t - \frac{\log t}{2\pi \Delta} \int_{1/2}^{\infty} \log\big(1 + e^{-2\pi \Delta (\sigma - 1/2)}\big) \>\dsigma + O\left(\frac{e^{-2\pi \Delta}\log t }{\pi\Delta}\right)+O(1)
 \\
 &= \frac{1}{4} \log t - \frac{\log t}{2\pi^2 \Delta^2} \int_{0}^{\infty} \log\big(1 + e^{-2\alpha}\big) \>\dalpha + O\left(\frac{ e^{-2\pi \Delta}\log t }{\pi\Delta}\right)+ O(1).
 \end{split}
 \end{align}
By \cite[\textsection 4.291, Formula 1]{GR}, we have
 \begin{equation}\label{em5}
 \int_0^\infty \log\big(1 \!+\! e^{-2\alpha} \big) \ \! \dalpha = \frac{1}{2} \int_0^1 \frac{\log(1\!+\!u)}{u} \ \! \du = \frac{\pi^2}{24}.
 \end{equation}
Therefore, by combining the estimates \eqref{em0}, \eqref{em4}, \eqref{useExplicit:Lower}, and \eqref{em5}, we see that
$$ S_1(t) \leq \frac{\log t}{48 \pi \Delta^2} + O\left(\frac{ e^{-2\pi \Delta}\log t }{\pi\Delta}\right)+ O\left(e^{\pi \Delta} +1 \right) + O\left( \frac{\Delta^2 e^{\pi \Delta}}{1+\Delta t} \right).$$
Choosing $\pi \Delta = \log \log t - 3 \log \log \log t$ in the above inequality, we obtain 
$$S_1(t) \leq \frac{\pi}{48} \frac{\log t}{(\log\log t)^2} + O\left( \frac{\log t \log\log\log t}{(\log\log t)^3} \right).$$
This proves the upper bound for $S_1(t)$ in Theorem \ref{th2}.
\smallskip

Again by Stirling's formula, parts (i) and (iii) of Lemma \ref{lem:minorant}, and (\ref{pi}), we see that
 \begin{align} \label{useExplicit:Upper}
 \begin{split} 
\frac{1}{2\pi }& \int_{-\infty}^{\infty} h^+_\Delta(u) \, \text{Re } \frac{\Gamma^{\prime}}{\Gamma} \left( 
 \frac{1}{4} + \frac{iu}{2} \right) \du 
 \\
 &= \frac{1}{2\pi} \int_{-\infty}^{\infty} g_\Delta^+(u) \Big(\log t + O\big(\log (2\!+\!|u|)\big)\Big) \> \du 
 \\
 &= \frac{1}{2\pi} \int_{-\infty}^{\infty} \Big\{ f_1(u) + \big(g_\Delta^+(u) - f_1(u)\big)\Big\} \Big(\log t + O\big(\log (2\!+\!|u|)\big)\Big) \> \du
 \\
 &= \frac{1}{4} \log t - \frac{\log t}{2\pi \Delta} \int_{1/2}^{3/2} \left\{\log\big(1 - e^{-2\pi \Delta (\sigma - 1/2)}\big)- \log \big(1 - e^{-2\pi\Delta}\big)\right\} \>\dsigma + O(1)
 \\
 &\leq \frac{1}{4} \log t - \frac{\log t}{2\pi \Delta} \int_{1/2}^{\infty} \log\big(1 - e^{-2\pi \Delta (\sigma - 1/2)}\big) \>\dsigma + O\left(\frac{e^{-2\pi \Delta}\log t }{\pi\Delta}\right)+O(1)
 \\
 &= \frac{1}{4} \log t - \frac{\log t}{2\pi^2 \Delta^2} \int_{0}^{\infty} \log\big(1 - e^{-2\alpha}\big) \>\dalpha + O\left(\frac{ e^{-2\pi \Delta}\log t }{\pi\Delta}\right)+ O(1).
 \end{split}
 \end{align}
By \cite[\textsection 4.291, Formula 2]{GR} we have
 \begin{equation}\label{em6}
  \int_0^\infty \log\big(1 \!-\! e^{-2\alpha} \big) \ \! \dalpha = \frac{1}{2} \int_0^1 \frac{\log(1\!-\!u)}{u} \ \! \du = -\frac{\pi^2}{12}.
  \end{equation}
By combining estimates \eqref{em0}, \eqref{em4}, \eqref{useExplicit:Upper}, and \eqref{em6}, it follows that
 $$ S_1(t)  \geq - \frac{\log t}{24 \pi \Delta^2} + O\left(\frac{ e^{-2\pi \Delta}\log t }{\pi\Delta}\right)+ O\left(e^{\pi \Delta} +1 \right) + O\left( \frac{\Delta^2 e^{\pi \Delta}}{1+\Delta t} \right).$$
Again, choosing $\pi \Delta = \log \log t - 3 \log \log \log t$, we obtain 
$$S_1(t) \geq -\frac{\pi}{24} \frac{\log t}{(\log\log t)^2} + O\left( \frac{\log t \log\log\log t}{(\log\log t)^3} \right).$$
This proves the required lower bound for $S_1(t)$, and therefore completes the proof of Theorem \ref{th2}.
\end{proof}


\section{Theorem \ref{th2} implies Theorem \ref{th1}}

In this section, we deduce Theorem \ref{th1} from Theorem \ref{th2} and the following lemma.\footnote{Maksym Radziwill (personal communication) has independently obtained a version of this lemma with the constants $\pm 1/4\pi$ replaced by $\pm 1/2\pi$.}

\begin{lemma}\label{lemma1}
Let $t\geq 2$ and $0< h \leq \sqrt{t}$. Then
$$-\frac{h \log t}{4\pi}+\frac{1}{h} \int_{t-h}^t S(u) \ \! \du + \vartheta_{h,t}^- \ \! \leq \ \! S(t) \ \! \leq \ \! \frac{h \log t}{4\pi}+ \frac{1}{h} \int_{t}^{t+h} S(u) \ \! \du + \vartheta_{h,t}^+ $$
where $ \vartheta_{h,t}^\pm = O( h)  + O(1/t).$
\end{lemma}

\begin{proof}
We first establish the upper bound for $S(t)$ in the lemma.  Let $0\leq \nu \leq \sqrt{t}$. Then, by (\ref{N(T)}), we have
$$ 0 \leq N(t\!+\!\nu)-N(t) = \frac{\nu}{2\pi}\log t +O(\nu) +O\left(\frac{1}{t}\right) + S(t\!+\!\nu)-S(t)$$
or, upon rearranging terms, 
$$ S(t) \leq S(t\!+\!\nu) + \frac{\nu}{2\pi}\log t +O(\nu) +O\left(\frac{1}{t}\right).$$
Now, for $0\leq h \leq \sqrt{t}$, it follows that
\begin{eqnarray*}
 S(t) &=& \frac{1}{h} \int_{0}^{h} S(t) \ \!\dnu 
 \\
 &\leq& \frac{1}{h} \int_{0}^h \left(  S(t\!+\!\nu) + \frac{\nu}{2\pi}\log t +O(\nu) +O\left(\frac{1}{t}\right)\right) \dnu 
 \\
 &=&  \frac{1}{h} \int_{t}^{t+h} S(u) \ \!\du +  \frac{h \log t}{4\pi} + O(h) +O\left(\frac{1}{t}\right),
 \end{eqnarray*}
as claimed.

The proof of the lower bound for $S(t)$ in the lemma is similar. Let $0\leq \nu \leq \sqrt{t}$ and observe that
$$ 0 \leq N(t)-N(t\!-\!\nu) = \frac{\nu}{2\pi}\log t +O(\nu) +O\left(\frac{1}{t}\right) + S(t)-S(t\!-\!\nu)$$
or, upon rearranging terms, that
$$ S(t) \geq S(t\!-\!\nu) -  \frac{\nu}{2\pi}\log t +O(\nu) +O\left(\frac{1}{t}\right).$$
Hence, for $0\leq h \leq \sqrt{t}$, we have
\begin{eqnarray*}
S(t) &=& \frac{1}{h} \int_{0}^{h} S(t) \ \!\dnu 
 \\
 &\geq& \frac{1}{h} \int_{0}^h \left(  S(t\!-\!\nu) - \frac{\nu}{2\pi}\log t +O(\nu) +O\left(\frac{1}{t}\right)\right) \dnu 
 \\
 &=&  \frac{1}{h} \int_{t-h}^{t} S(u) \ \!\du -  \frac{h \log t}{4\pi} + O(h) +O\left(\frac{1}{t}\right),
 \end{eqnarray*}
as claimed. This completes the proof of the lemma.
\end{proof}


\begin{proof}[Proof of Theorem \ref{th1}]  
The proof of Theorem \ref{th1} relies on the estimate
\begin{equation}\label{ineq1}
 \big| S_1(t\pm h) \! - \! S_1(t) \big| \ \leq \ \frac{\pi}{16} \frac{\log t}{(\log\log t)^2} + O\left(\frac{\log t \log\log\log t}{(\log\log t)^3} \right)
\end{equation}
which, for $0<h \leq \sqrt{t}$, follows immediately from Theorem \ref{th2}. 
\smallskip

First we prove the upper bound for $S(t)$ implicit in Theorem \ref{th1}. Evidently,
$$ \int_{t}^{t+h} S(u) \ \!\du = S_1(t\!+\!h) \! - \! S_1(t) \leq \frac{\pi}{16} \frac{\log t}{(\log\log t)^2} + O\left(\frac{\log t \log\log\log t}{(\log\log t)^3} \right).$$
Therefore, by the upper bound in Lemma \ref{lemma1}, it follows that
$$ S(t) \leq \frac{\pi}{16 h} \frac{\log t}{(\log\log t)^2} + \frac{h}{4 \pi}\log t + O\left(\frac{\log t \log\log\log t}{h(\log\log t)^3} \right)+ O(h) +O\left(\frac{1}{t}\right).$$
Choosing $h= \frac{\pi}{2}\frac{1}{\log\log t}$, which minimizes the main term on the right-hand side of the above inequality, we deduce that
$$ S(t) \leq \frac{1}{4} \ \! \frac{\log t}{ \log \log t} +  O\left(\frac{\log t \log\log\log t}{(\log\log t)^2} \right).$$

The lower bound for $S(t)$ implicit in Theorem \ref{th1} can be established in a similar manner. By (\ref{ineq1}), we have
$$ \int_{t-h}^{t} S(u) \ \!\du = S_1(t) \! - \! S_1(t\!-\!h) \geq -\frac{\pi}{16} \frac{\log t}{(\log\log t)^2} + O\left(\frac{\log t \log\log\log t}{(\log\log t)^3} \right)$$
and hence
$$ S(t) \geq -\frac{\pi}{16 h} \frac{\log t}{(\log\log t)^2} - \frac{h}{4 \pi}\log t + O\left(\frac{\log t \log\log\log t}{h(\log\log t)^3} \right)+ O(h) +O\left(\frac{1}{t}\right)$$
using the lower bound in Lemma \ref{lemma1}. Choosing $h= \frac{\pi}{2}\frac{1}{\log\log t}$, which maximizes the main term on the right-hand side of the above inequality, it follows that
$$ S(t) \geq -\frac{1}{4} \ \! \frac{\log t}{ \log \log t} + O\left(\frac{\log t \log\log\log t}{(\log\log t)^2} \right).$$
The theorem now follows by combining these estimates.
\end{proof}

\noindent{\it Remark.} Using a different method, Fujii \cite{F} has shown that 
$$  (-0.50903 + o(1) ) \frac{\log t}{(\log \log t)^2} \ \leq \ S_1(t) \ \leq \ (0.31252 + o(1) ) \frac{\log t}{(\log \log t)^2}.$$
These bounds were stated, in a slightly different form, in the introduction. Using Lemma \ref{lemma1}, the method above, and Fujii's bounds for $S_1(t)$, it can be shown that $$ |S(t)| \leq (0.51138  + o(1)) \frac{\log t}{\log \log t},$$
which is almost as strong as the bound that Goldston and Gonek obtained in \cite{GG}.


\section{Proof of Theorem \ref{th1} via extremal functions} \label{Sec:StviaExtremal}
In this section, we prove Theorem \ref{th1} using families of majorants and minorants of exponential type $2\pi \Delta$ for the function  
\begin{equation*}
f(x) = \arctan\left(\frac{1}{x}\right) - \frac{x}{1 + x^2}.
\end{equation*}
The properties of the extremal functions required in the proof are described in the next lemma.

\begin{lemma}
\label{lem:extremalfunctions}
Let $f(x) = \arctan\left(1/x\right) - x/(1 + x^2) $ and let $\Delta \geq 1$.  Then there are unique real entire functions $m_{\Delta}^-: \C \to \C$ and $m_{\Delta}^+: \C \to \C$ satisfying the following properties: 
\begin{itemize}
\item[(i)]  For all real $x$ we have 
 \begin{equation}\label{ineq:boundforgDeltax}
 \frac{-C}{1+x^2} \le m_{\Delta}^-(x) \le f(x) \leq m_{\Delta}^+(x) \leq  \frac{C}{1+x^2},
\end{equation}
for some positive constant $C$.  Moreover, for any complex number 
 $z=x+iy$ we have 
 \begin{equation} \label{exp_bound_g}
 \big|m_{\Delta}^{\pm}(z)\big| \ll  \frac{\Delta^2  }{(1\!+\! \Delta |z|)}e^{2\pi \Delta|y|}.
  \end{equation}

 \smallskip
 
 \item[(ii)]  The Fourier transforms of $m_{\Delta}^\pm$, namely 
\begin{equation*}
 {\widehat m}_{\Delta}^\pm(\xi) = \int_{-\infty}^{\infty}  m_{\Delta}^\pm(x)\, e^{-2\pi i x \xi }\,\dx,  
 \end{equation*}
 are continuous functions supported on the interval $[-\Delta, \Delta]$ and satisfy 
 \begin{equation*}
 \big|{\widehat m}_{\Delta}^\pm(\xi)\big| \ll 1
 \end{equation*}
 for all $\xi \in [-\Delta, \Delta]$, where the implied constant is independent of $\Delta$. 
 \smallskip
 
\item[(iii)]  The $L^1$-distances of $m^\pm_\Delta$ to $f$ are given by
 \begin{equation*}
  \int_{-\infty}^{\infty} \big\{m_{\Delta}^+(x) - f(x)\big\}\, \dx  = \int_{-\infty}^{\infty} \big\{f(x)-m_{\Delta}^-(x)\big\} \,\dx = \frac{\pi}{2\Delta} .\\
\end{equation*}
 \end{itemize}
\end{lemma}

We postpone the proof of Lemma \ref{lem:extremalfunctions} until the next section and proceed with the proof of Theorem \ref{th1}.

\begin{proof}[Proof of Theorem \ref{th1}]

Let $\Delta\geq 1$. By Lemma \ref{S(t)sum} and (i) of Lemma \ref{lem:extremalfunctions}, we observe that 
\begin{equation}\label{Sec2.6}
\frac{1}{\pi}  \sum_{\gamma} h^-(\gamma) + O(1) \ \! \leq  \ \! S(t)  \ \!  \leq  \ \! \frac{1}{\pi}  \sum_{\gamma} h^+(\gamma) + O(1)
\end{equation}
where $h^{\pm}(z) = m_{\Delta}^{\pm}(t - z)$ and thus $\widehat{h}^{\pm}(\xi) = \widehat{m}_{\Delta}^{\pm}(-\xi) e^{-2\pi i \xi t}$. From (i) of Lemma \ref{lem:extremalfunctions} we find that 
\begin{equation}\label{Sec2.7}
\left|h^\pm\left(\frac{1}{2i}\right)\right| + \left|h^\pm\left(-\frac{1}{2i}\right) \right|\ll \frac{\Delta^2 e^{\pi \Delta}}{(1+\Delta t)}.
\end{equation} 
 Using (ii) of  Lemma \ref{lem:extremalfunctions},  we have
 \begin{equation}\label{Sec2.8}
\left|\widehat{h}^\pm(0)\right| \ll 1.
 \end{equation}  
 Using Stirling's formula for $\Gamma'(u)/\Gamma(u)$, parts (i) and (iii) of Lemma \ref{lem:extremalfunctions}, and the fact that 
 $\int_{-\infty}^{\infty}f(x)\,\dx = 0$, it follows that
 \begin{align} \label{Sec2.9.5}
 \begin{split}
 & \frac{1}{2\pi } \int_{-\infty}^{\infty} h^\pm(u) \,\text{Re} \frac{\Gamma^{\prime}}{\Gamma} \left( 
 \frac{1}{4} + \frac{iu}{2} \right) \, \du \\
 & \ \ \ \ \ \ \ = \frac{1}{2\pi} \int_{-\infty}^{\infty} m_\Delta^\pm(u) \big(\log t + O(\log (2+|u|))\big) \, \du  
 \\
 & \ \ \ \ \ \ \  = \pm \frac{\log t}{4\Delta} + O(1).
 \end{split}
 \end{align}
Finally, using (ii) of  Lemma \ref{lem:extremalfunctions}, the sum over prime powers satisfies the inequality
\begin{align}\label{Sec2.10}
\begin{split}
\left|\frac{1}{2\pi}\sum_{n=2}^{\infty} \frac{\Lambda(n)}{\sqrt{n}}\left(\widehat{h}^\pm\left( \frac{\log n}{2\pi }\right) + \widehat{h}^\pm\left( \frac{-\log n}{2\pi }\right)\right)\right| \ll \sum_{n\leq e^{2\pi \Delta}} \frac{\Lambda(n)}{\sqrt{n}} \ll e^{\pi \Delta}.
\end{split}
\end{align}
Combining \eqref{Sec2.6} with the bounds \eqref{Sec2.7}, \eqref{Sec2.8}, \eqref{Sec2.9.5},  and \eqref{Sec2.10} gives 
\begin{equation*}\label{Sec2.11}
|S(t)| \leq \frac{\log t}{4 \pi \Delta} + O\left( e^{\pi \Delta} + 1\right) + O \left(\frac{\Delta^2 e^{\pi \Delta}}{1+\Delta t} \right).
\end{equation*}
Choosing $\pi \Delta = \log \log t - 3\log \log \log t$ in the above inequality, we obtain 
\begin{equation*}
|S(t)| \leq \ \frac{1}{4} \frac{\log t}{\log \log t} + O\left(\frac{\log t \log \log \log t}{(\log \log t)^2}\right),
\end{equation*}
which completes the proof of Theorem \ref{th1}.
\end{proof}


\section{Extremal functions} \label{Sec:Extremal}
In this section we discuss the extremal functions used in this paper, and in particular we prove Lemma \ref{lem:minorant} and Lemma \ref{lem:extremalfunctions}. Let us start with a brief description of the extremal problem and an account of its development.

\subsection{The Beurling-Selberg extremal problem} We say that an entire function $K:\C \to \C$ has {\it exponential type} at most $2\pi \Delta$ if, for every $\epsilon >0$, there exists a positive constant $C_{\epsilon}$, such that the inequality 
\begin{equation*}\label{intro0}
|K(z)| \leq C_{\epsilon} e^{(2\pi \Delta + \epsilon) |z|}
\end{equation*}
holds for all $z \in \C$. The extremal problem we are interested in here is the following: given a function $F: \R \to \R$, and a constant $\Delta >0$, we seek an entire function $K(z)$ of exponential type at most $2\pi\Delta$ such that the integral
\begin{equation}\label{BS1}
 \int_{-\infty}^{\infty} |F(x) - K(x)|\, \dx
\end{equation}
is minimized. This is a classical problem in harmonic analysis and approximation theory, considered by Bernstein, Akhiezer, Krein, Nagy, and others, dating back to at least 1938 (see for instance \cite{K,Na}). For applications to analytic number theory, it is convenient to consider a one-sided variant of this problem in which we ask additionally that $K(z)$ is real on $\R$ and that $K(x) \geq F(x)$ for all $x \in \R$. In this case, a minimizer of the integral \eqref{BS1} is called an extremal majorant of $F(x)$. Extremal minorants are defined analogously. Beurling independently started working on this extremal one-sided problem in the late 1930s, and obtained the solution for $F(x) = \sgn(x)$ and an inequality for almost periodic functions in an unpublished manuscript. Later, Selberg \cite{S2} recognized that the one-sided extremals for the signum function could produce majorants and minorants for characteristic functions of intervals, and used this fact to obtain a sharp form of the large sieve inequality. For a historical account of the early developments of this theory we refer to J. D. Vaaler's classical paper \cite{V}. Since here we are mainly interested in the one-sided version of this problem, we refer to it as the {\it Beurling-Selberg extremal problem}. 
\smallskip

The applications to number  theory rely heavily on the fact that these extremal functions have distributional Fourier transforms compactly supported in the interval $[-\Delta,\Delta]$, as a consequence of the Paley-Wiener theorem. An account of these applications must certainly include Hilbert-type inequalities \cite{CV2, GV, Lit, V}, Erd\"{o}s-Tur\'{a}n discrepancy inequalities \cite{CV2, LV, V}, optimal approximations of periodic functions by trigonometric polynomials \cite{CV2, CV3, V}, Tauberian theorems \cite{GV}, higher dimensional diophantine problems \cite{BMV, HV}, and more recently, bounds for the Riemann zeta-function under the Riemann Hypothesis \cite{CC, CS, GG}.
\smallskip

The subtle point of this theory is that given any $F:\R \to \R$ there is no general technique which is known to produce a solution of this extremal problem. 
There are, however, other examples of families for which the solution has been achieved. This includes the exponential functions $F(x) = e^{-\lambda |x|}$, $\lambda>0$, by Graham and Vaaler in \cite{GV}, and the odd and truncated power functions $F(x) = x^n \sgn(x)$ and $F(x) = (x^+)^n$, $n\in \N$, by Littmann in \cite{Lit, Lit2}. Later, using an exponential subordination, Carneiro and Vaaler in  \cite{CV2,CV3} were able to extend the construction of extremal approximations for a class of even functions that includes $F(x) = \log|x|$ and $F(x) = \log\bigl((x^2+4)/x^2 \bigr)$ (the latter function played an important role in the work of Chandee and Soundararajan \cite{CS}). Finally, the recent works \cite{CL,CLV} provide the latest tools to generate solutions of the Beurling-Selberg extremal problem for families of even, odd, and truncated functions via a certain Gaussian subordination. In particular, we shall find our extremal majorants and minorants for $f_1(x)$ and $f(x)$ in this framework, as described below.
\smallskip


\subsection{Proof of Lemma \ref{lem:minorant}} In the work \cite{CLV}, Carneiro, Littmann, and Vaaler solve the Beurling-Selberg extremal problem for the class of even functions $g: \mathbb{R} \rightarrow \mathbb{R}$ given by
$$g(x) = \int_0^{\infty} e^{-\pi\lambda x^2} \> \dnu(\lambda), $$
where $\nu(\lambda)$ is a finite non-negative measure Borel measure on $(0, \infty)$. It turns out that our function $f_1(x) = 1 - x\arctan(1/x)$ is included in this class.
\smallskip

In fact, for $\Delta \geq 1$, we define the non-negative measure
$$ \dnu_{\Delta}(\lambda) := \int_{1/2}^{3/2}\left(\frac{e^{-\pi \lambda (\sigma - 1/2)^2\Delta^2} - e^{-\pi \lambda \Delta^2}}{2\lambda}\right)\, \dsigma \, \dl\,,$$
and set
 \begin{eqnarray*} \label{eqn:measureforfunction}
F_{\Delta}(x) &:=& \int_0^{\infty} e^{-\pi \lambda x^2} \> \dnu_{\Delta}(\lambda).
\end{eqnarray*}
From \new{\cite[Section 11]{CLV}} we have 
\begin{equation*} \label{eqn:Ax}
 A(x) := \frac{1}{2}\log \left( \frac{x^2 + \Delta^2}{x^2 + (\sigma - 1/2)^2\Delta^2}\right) = \int_0^{\infty} e^{-\pi\lambda x^2} \left(\frac{e^{-\pi \lambda (\sigma - 1/2)^2\Delta^2} - e^{-\pi \lambda\Delta^2}}{2\lambda} \right)\> \dl.
\end{equation*}
Integrating $A(x)$ from $\sigma=1/2$ to $\sigma=3/2$ , we derive that 
$$ F_{\Delta}(x) =  1 - \frac{x}{\Delta}\arctan\left(\frac{\Delta}{x}\right).$$
In particular, this shows that the measure $\dnu_\Delta(\lambda)$ is finite on $(0,\infty)$ since
$$ \int_0^{\infty} \dnu(\lambda) = F_{\Delta}(0) = 1,$$ 
and we observe that the function $f_1(x)$ in Lemma \ref{Sec:Extremal} satisfies
\begin{equation} \label{eqn:fF}
 f_1(x) =  F_{\Delta}(\Delta x).
\end{equation}

By \new{ \cite[Corollary 17]{CLV}}, there is both a unique extremal minorant $G_{\Delta}^-(z)$ and a unique extremal majorant $G_{\Delta}^+(z)$ of exponential type $2\pi$ for $F_{\Delta}(x),$ and these functions are of the form
\begin{equation*} \label{eqn:sumofGDelta}
G_{\Delta}^-(z) = \left( \frac{\cos \pi z}{\pi} \right)^2 \sum_{n=-\infty}^{\infty} 
 \left\{\frac{F_{\Delta}\bigl(n-\frac 12\bigr)}{\bigl(z-n+\frac 12\bigr)^2} + \frac{F_{\Delta}^{\prime}\bigl(n-\frac 12\bigr)}{\bigl(z-n+\frac 12\bigr)}\right\}
\end{equation*}
and 
\begin{equation*} \label{eqn:sumofMDelta}
G_{\Delta}^+(z) = \left( \frac{\sin \pi z}{\pi} \right)^2  
 \left\{ \sum_{n=-\infty}^{\infty} \frac{F_{\Delta}(n)}{(z-n)^2} + \sum_{n \neq 0} \frac{F_{\Delta}^{\prime}(n)}{(z-n)}\right\},
\end{equation*}
respectively. The observation in (\ref{eqn:fF}) suggests choosing
\begin{equation}\label{defofg}
g_{\Delta}^-(z) = G_{\Delta}^-(\Delta z) \,\,\,\,\,\, {\rm and} \,\,\,\,\,\,g_{\Delta}^+(z) = G_{\Delta}^+(\Delta z) 
\end{equation}
for the extremal functions $g_{\Delta}^\pm(z)$ in Lemma \ref{lem:minorant}.

\begin{figure}
\includegraphics[scale=.75]{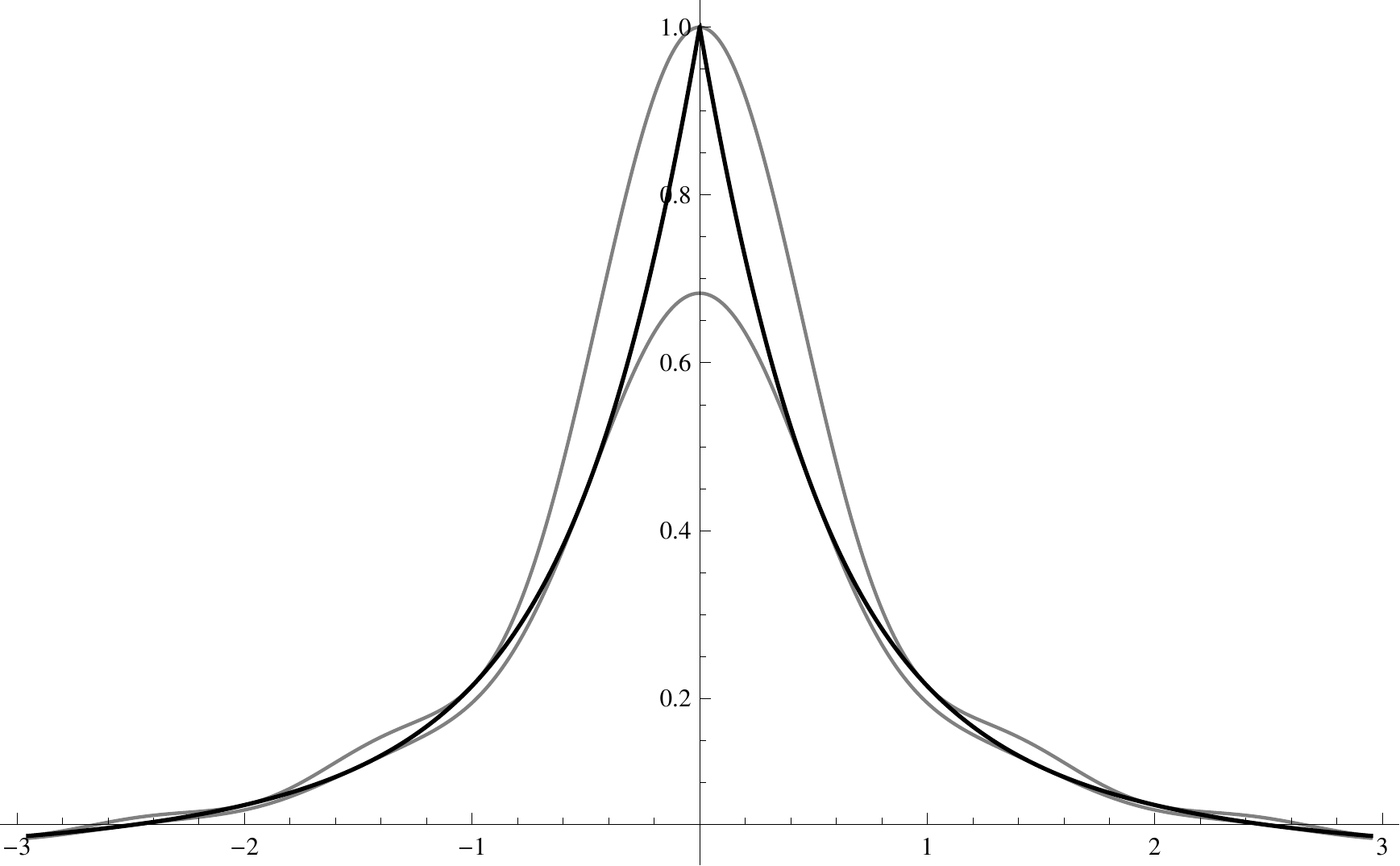} 
\caption{The function $f_1(x)=1-x\arctan(1/x)$ plotted with its real entire majorant and minorant of exponential type $2\pi$.}
\end{figure}


\subsubsection*{Proof of part (i) of Lemma \ref{lem:minorant}}  First, we show that there are constants $A$ and $B$ such that
\begin{equation} \label{ineq:f1}
|f_1(x)| \leq \frac{A}{1+ x^2}
\end{equation}
and 
\begin{equation} \label{ineq:fprime1}
|f'_1(x)| \leq \new{\frac{B}{|x|(1+x^2) }}.
\end{equation}
To prove both inequalities, we consider two cases: $|x| \leq 2$ and $|x| > 2.$ For the first case, both (\ref{ineq:f1}) and (\ref{ineq:fprime1}) hold since $|f_1(x)|$ and $|f'_1(x)|$ are bounded. For the second case, $|x|>2$, we write $f_1(x)$ and $f'_1(x)$ as power series 
\begin{equation}\label{eqn:powerseriesofF1}
f_1(x) = \sum_{n=1}^{\infty} (-1)^{n-1}\frac{1}{(2n + 1)x^{2n}} \quad \text{ and } \quad 
 f'_1(x) = \sum_{n = 1}^{\infty} (-1)^{n} \frac{2n}{(2n + 1)x^{2n + 1}}.
\end{equation}
It is not difficult to see that the bounds in (\ref{ineq:f1}) and (\ref{ineq:fprime1}) follow in this case from the series expansions in (\ref{eqn:powerseriesofF1}). 
\smallskip

We are now in position to establish (\ref{ineq:boundforgDeltazf1}).  Observe that 
\begin{equation} \label{eqn:sumofGDelta1}
G_{\Delta}^-(z) = \sum_{n=-\infty}^{\infty} \left(\frac{\sin \pi \bigl(z-n + \frac 12\bigr)}{\pi\bigl(z-n + \frac 12\bigr)}\right)^2 
 \left\{f_1\left(\frac{n-\tfrac 12}{\Delta}\right) + \frac{\bigl(z-n + \frac 12\bigr)}{\Delta} f_1^{\prime}\left(\frac{n-\tfrac 12}{\Delta}\right)\right\},
\end{equation}
and that
\begin{equation} \label{eqn:sumofGplusDelta}
 G_{\Delta}^+(z) = \sum_{n \neq 0}  \left( \frac{\sin \pi (z-n)}{\pi (z-n) }\right)^2 \left\{f_1\left(\frac{n}{\Delta}\right) +  \frac{(z-n)}{\Delta}f'_1\left(\frac{n}{\Delta}\right)\right\} + \left( \frac{\sin \pi z}{\pi z }\right)^2 .
\end{equation}
Moreover, for any complex number $\xi$ we have $|\sin (\pi \xi)/ (\pi \xi)|^2 \ll e^{2 \pi |{\rm Im} \xi|}/ (1 + |\xi|^2)$. Therefore from (\ref{ineq:f1}), (\ref{ineq:fprime1}), (\ref{eqn:sumofGDelta1}), and (\ref{eqn:sumofGplusDelta}), we can conclude that
\begin{equation*}
 |G_{\Delta}^\pm(x+iy)| \ll  \frac{\Delta^2  }{1+ |x+iy|}e^{2\pi |y|}.
\end{equation*}
This follows, for instance, by bounding trivially using absolute values and splitting the sums into $n\leq |z|/2$ and $n > |z|/2$. Hence from (\ref{defofg}) we arrive at (\ref{ineq:boundforgDeltazf1}).
\smallskip

Next, we establish (\ref{ineq:boundforgDeltaxf1}). From the choice of $g_{\Delta}^\pm$, for all $x$, we obtain that 
$$ g_{\Delta}^-(x) \leq f_1(x) \leq g_{\Delta}^+(x).$$
For a real number $x$, we have $f_1(x)\geq 0$, $f^{\prime}_1(-x) = -f^{\prime}_1(x)$, and $f_1'(x) \leq 0$ if $x>0$. Pairing the terms $n\geq 1$ and $1-n\leq0$ in the sum on the right-hand side of (\ref{eqn:sumofGDelta1}), we obtain that
\begin{align}\label{2boundforg}
\begin{split}
G_{\Delta}^-(x) &\geq \left( \frac{\cos \pi x}{\pi }\right)^2 \sum_{n=1}^{\infty} \frac{1}{\Delta} f_{1}^{\prime}\left(\frac{n-\tfrac 12}{\Delta}\right)\left\{ \frac{1}{\bigl(x - n + \tfrac 12\bigr)} - \frac{1}{\bigl(x + n - \tfrac12\bigr)}\right\}\\
& = \sum_{n=1}^{\infty} \frac{\sin^2 \pi\bigl(x -n +\tfrac12\bigr)}{\pi^2 \Bigl(x^2 - \bigl(n-\tfrac 12\bigr)^2\Bigr)}\, \frac{2 \bigl( n - \tfrac12 \bigr)}{\Delta}\, f_{1}^{\prime}\left(\frac{n-\tfrac 12}{\Delta}\right)\\
& \geq \sum_{(n-1/2)\leq |x|} \frac{\sin^2 \pi\bigl(x -n +\tfrac12\bigr)}{\pi^2 \Bigl(x^2 - \bigl(n-\tfrac 12\bigr)^2\Bigr)}\, \frac{2 \bigl( n - \tfrac12 \bigr)}{\Delta}\, f_{1}^{\prime}\left(\frac{n-\tfrac 12}{\Delta}\right),
\end{split}
\end{align}
Using (\ref{ineq:fprime1}) and (\ref{2boundforg}) \new{(splitting the sum into $n \leq |x|/2$ and $n> |x|/2$}), it follows that there is a constant $C$ such that 
\begin{equation*}
-C \frac{\Delta^2}{\Delta^2 + x^2} \leq G_{\Delta}^-(x)\,,
\end{equation*}
and thus from (\ref{defofg}) we derive the left most inequality of (\ref{ineq:boundforgDeltaxf1}).
\smallskip

For the right most inequality in (\ref{ineq:boundforgDeltaxf1}), we again use the fact that for real numbers $x$ we have $f_1(x) = f_1(-x)$, $f_1^{\prime}(-x) = -f_1^{\prime}(x)$,  and $f_1'(x) \leq 0$ if $x>0$. So pairing the terms $n\geq 1$ and $-n \leq - 1$ in the sum on the right-hand side of (\ref{eqn:sumofGplusDelta}), we see that
\begin{align}\label{2boundforM}
\begin{split}
G_{\Delta}^+(x) &=  \sum_{n =1}^{\infty}\left(\frac{\sin^2 \pi\bigl(x -n\bigr)}{\pi^2 \bigl(x^2 - n^2\bigr)^2}\right)f_1\left(\frac{n}{\Delta}\right) (2x^2 + 2n^2) + \left(\frac{\sin \pi x}{\pi x} \right)^2\\
& \ \ \ \ \ + \sum_{n=1}^{\infty} \frac{\sin^2 \pi\bigl(x -n \bigr)}{\pi^2 \bigl(x^2 - n^2\bigr)}\, \frac{2 n}{\Delta}\, f_1^{\prime}\left(\frac{n}{\Delta}\right)\\
& \leq  \sum_{n =1}^{\infty}\left(\frac{\sin^2 \pi\bigl(x -n\bigr)}{\pi^2 \bigl(x^2 - n^2\bigr)^2}\right)f_1\left(\frac{n}{\Delta}\right) (2x^2 + 2n^2) + \left(\frac{\sin \pi x}{\pi x} \right)^2\\
& \ \ \ \ \ + \sum_{n\geq |x|} \frac{\sin^2 \pi\bigl(x -n \bigr)}{\pi^2 \bigl(x^2 - n^2\bigr)}\, \frac{2 n}{\Delta}\, f_1^{\prime}\left(\frac{n}{\Delta}\right).\\
\end{split}
\end{align}
Using (\ref{ineq:f1}), (\ref{ineq:fprime1}) and (\ref{2boundforM}), it follows that there is a constant $C$ (\new{again one can break the sum into $n \leq |x|/2$ and $n> |x|/2$ to verify this}) such that
$$G_{\Delta}^+(x) \leq C \frac{\Delta^2}{\Delta^2 + x^2}.$$
Thus, from (\ref{defofg}), we deduce the upper bound for $g_{\Delta}^+(x)$  in (\ref{ineq:boundforgDeltaxf1}). This completes the proof of part (i) of Lemma \ref{lem:minorant}.


\subsubsection*{Proof of part (ii) of Lemma \ref{lem:minorant}}  From the proof of part (i) of Lemma \ref{lem:minorant}, it follows that the functions $g_{\Delta}^\pm$ have exponential type $2\pi \Delta$ and are integrable  on $\mathbb{R}.$ Therefore, by the Paley-Wiener theorem, the Fourier transforms $\widehat{g}_{\Delta}^\pm$ are compactly supported on the interval $[-\Delta, \Delta]$. \new{Moreover, using  \eqref{ineq:boundforgDeltaxf1} we have
\begin{align*}
\begin{split}
\big|\widehat{g}_{\Delta}^{\pm}(\xi)\big| &= \left| \int_{-\infty}^{\infty} g_{\Delta}^\pm(x) e^{-2\pi i x \xi }\, \>\dx\right| \leq \int_{-\infty}^{\infty} \big|g_{\Delta}^\pm(x) \big|\,\dx \leq C \int_{-\infty}^{\infty} \frac{1}{1+x^2}\,\dx \ll 1.
 \end{split}
\end{align*}
}


\subsubsection*{Proof of part (iii) of Lemma \ref{lem:minorant}} \new{From \cite[Section 11, Corollary 17 and Example 3]{CLV} we have
\begin{align*}
 &\int_{-\infty}^{\infty} \big\{ F_\Delta(x) - G_{\Delta}^-(x) \big\} \> \dx  =  \int_0^{\infty} \left\{ \sum_{\stackrel{n=-\infty}{n\neq0}}^{\infty} (-1)^{n+1}\, \lambda^{-1/2} \,e^{-\pi \lambda^{-1}n^2}\right\} \dnu_{\Delta}(\lambda)
  \\
& = \int_0^{\infty}\int_{1/2}^{3/2} \left\{ \sum_{\stackrel{n=-\infty}{n\neq0}}^{\infty} (-1)^{n+1}\, \lambda^{-1/2} \,e^{-\pi \lambda^{-1}n^2}\right\} \left(\frac{e^{-\pi \lambda (\sigma - 1/2)^2\Delta^2} - e^{-\pi \lambda \Delta^2}}{2\lambda}\right)\> \dsigma \, \dl\\
& = \int_{1/2}^{3/2} \int_0^{\infty}\left\{ \sum_{\stackrel{n=-\infty}{n\neq0}}^{\infty} (-1)^{n+1}\, \lambda^{-1/2} \,e^{-\pi \lambda^{-1}n^2}\right\} \left(\frac{e^{-\pi \lambda (\sigma - 1/2)^2\Delta^2} - e^{-\pi \lambda \Delta^2}}{2\lambda}\right)\>\dl\,  \dsigma\\
&=\int_{1/2}^{3/2} \log \left(\frac{1 + e^{-2\pi (\sigma - 1/2)\Delta}}{1 + e^{-2\pi\Delta}}\right)\,\dsigma,
\end{align*}
where the interchange of integrals is justified since the integrand is non-negative. In a similar way we have
\begin{align*}
 \int_{-\infty}^{\infty}& \big\{G_{\Delta}^+(x) -  F_\Delta(x) \big\} \> \dx  =  \int_0^{\infty} \left\{ \sum_{\stackrel{n=-\infty}{n\neq0}}^{\infty}  \lambda^{-1/2} \,e^{-\pi \lambda^{-1}n^2}\right\} \dnu_{\Delta}(\lambda)
  \\
& = \int_0^{\infty}\int_{1/2}^{3/2} \left\{ \sum_{\stackrel{n=-\infty}{n\neq0}}^{\infty} \lambda^{-1/2} \,e^{-\pi \lambda^{-1}n^2}\right\} \left(\frac{e^{-\pi \lambda (\sigma - 1/2)^2\Delta^2} - e^{-\pi \lambda \Delta^2}}{2\lambda}\right)\> \dsigma \, \dl\\
& = \int_{1/2}^{3/2} \int_0^{\infty}\left\{ \sum_{\stackrel{n=-\infty}{n\neq0}}^{\infty}  \lambda^{-1/2} \,e^{-\pi \lambda^{-1}n^2}\right\} \left(\frac{e^{-\pi \lambda (\sigma - 1/2)^2\Delta^2} - e^{-\pi \lambda \Delta^2}}{2\lambda}\right)\>\dl\,  \dsigma\\
&= - \int_{1/2}^{3/2} \log \left(\frac{1 - e^{-2\pi (\sigma - 1/2)\Delta}}{1 - e^{-2\pi\Delta}}\right)\,\dsigma.
\end{align*}
Part (iii) of Lemma \ref{lem:minorant} now follows from the change of variables (\ref{eqn:fF}) and \eqref{defofg}. 
}


\subsection{Proof of Lemma \ref{lem:extremalfunctions}}  
In this section we discuss the extremal functions that play a role in proving Lemma \ref{lem:extremalfunctions}. This is based on the recent work of Carneiro and Littmann \cite{CL} in which they solve the Beurling-Selberg extremal problem for the truncated (and odd) Gaussian and extend the construction to a class of truncated (and odd) functions which includes the function $f(x)=\arctan(1/x)-x/(1+x^2)$. In the discussion below, we focus only on the construction for odd functions.
\smallskip

Let $\lambda >0$ be a parameter and define the odd Gaussian  $x \mapsto G_{\lambda}^o(x)$ by
\begin{equation*}
G_{\lambda}^o(x) = \sgn(x)\,e^{- \pi \lambda x^2}.
\end{equation*}
In order to prove Lemma \ref{lem:extremalfunctions}, we consider the odd function 
\begin{equation*}
H_{\Delta}(x) = \arctan \left(\frac{\Delta}{x}\right) - \frac{\Delta x}{x^2 + \Delta^2}\,,
\end{equation*}
and observe that, for $x>0$, we have
\begin{equation*}
H_{\Delta}(x) = \int_0^{\infty} e^{-\lambda \pi x^2} \dmu_{\Delta}(\lambda),
\end{equation*}
where the measure \new{$\dmu_{\Delta}(\lambda)$} is given by 
\begin{equation}\label{measure}
\dmu_{\Delta}(\lambda) = \left\{\int_0^{\infty} \frac{t}{2\sqrt{\pi \lambda^3}} \,e^{-\tfrac{t^2}{4\lambda}} \left( \frac{1}{t} \sin (\sqrt{\pi}\Delta t) - \Delta\sqrt{\pi} \cos (\sqrt{\pi} \Delta t)\right) \dt\right\}\dl.
\end{equation}
Moreover, the measure $ \dmu_{\Delta}(\lambda)$ is a non-negative and finite Borel measure on $(0,\infty)$. These facts about the measure $ \dmu_{\Delta}(\lambda)$ are proved in Appendix A at the end of this article.
\smallskip

Evidently,
\begin{equation}\label{rep f}
f(x)  = H_{\Delta}(\Delta x).
\end{equation}
By \new{\cite[Theorem 4]{CL}}, in its odd version, there is a unique extremal minorant $M_{\Delta}^-(z)$ of exponential type $2\pi$ for $H_{\Delta}(x)$ and a unique extremal majorant $M_{\Delta}^+(z)$ of exponential type $2\pi$ for $H_{\Delta}(x)$. We let
\begin{equation}\label{def_g_m}
m_{\Delta}^+(z) = M_{\Delta}^+(\Delta z) \ \ \  {\rm and} \ \ \  m_{\Delta}^-(z) = M_{\Delta}^-(\Delta z).
\end{equation}
From \cite{CL}, we also have the following representations for these extremal functions:
\begin{equation}\label{rep G}
M_{\Delta}^-(z) = \Big( \frac{\sin \pi z}{\pi }\Big)^2  \sum_{\stackrel{n= -\infty}{n\neq 0}}^{\infty} 
 \left\{\frac{H_{\Delta}\bigl(n\bigr)}{\bigl(z-n\bigr)^2} + \frac{{H_{\Delta}}^{\prime}\bigl(n\bigr)}{(z-n)} - \frac{{H_{\Delta}}^{\prime}\bigl(n\bigr)}{z}\right\} - \frac{\pi}{2}  \Big( \frac{\sin \pi z}{\pi z }\Big)^2
\end{equation}
and 
\begin{equation}\label{rep M}
M_{\Delta}^+(z) = \Big( \frac{\sin \pi z}{\pi }\Big)^2  \sum_{\stackrel{n= -\infty}{n\neq 0}}^{\infty} 
 \left\{\frac{H_{\Delta}\bigl(n\bigr)}{\bigl(z-n\bigr)^2} + \frac{{H_{\Delta}}^{\prime}\bigl(n\bigr)}{(z-n)} - \frac{{H_{\Delta}}^{\prime}\bigl(n\bigr)}{z}\right\}\! +\! \frac{\pi}{2}  \Big( \frac{\sin \pi z}{\pi z}\Big)^2\!.
\end{equation}
Observe that $M_{\Delta}^-(z)  = -M_{\Delta}^+(-z)$ and thus 
\begin{equation}\label{Sec3.1}
m_{\Delta}^-(z)  = -m_{\Delta}^+(-z).
\end{equation}
Therefore, from \eqref{Sec3.1} it suffices to prove the statements of Lemma \ref{lem:extremalfunctions} for $m_{\Delta}^+$.

\begin{figure}
\includegraphics[scale=.75]{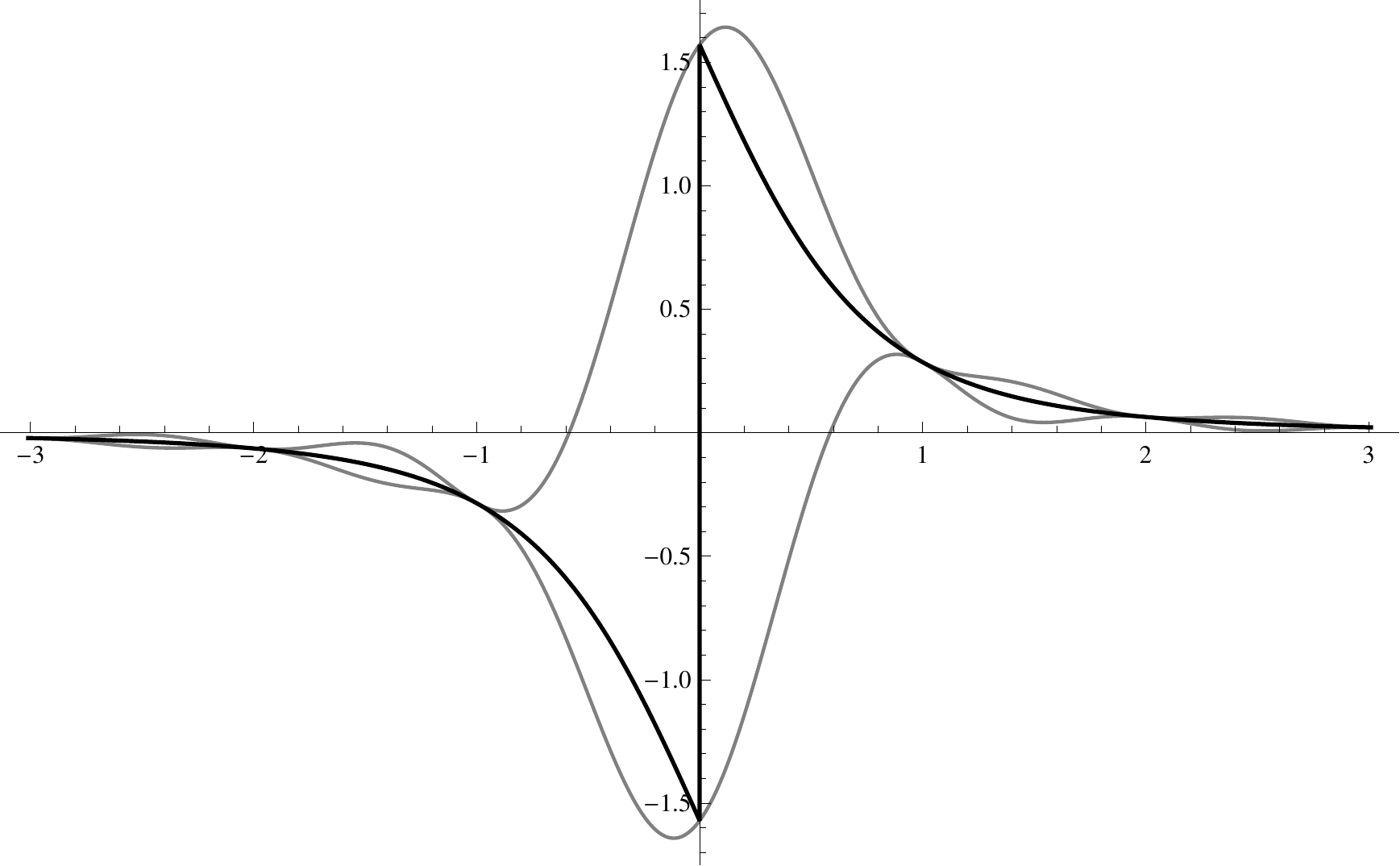} 
\caption{The function $f(x)=\arctan(1/x)-x/(1+x^2)$ plotted with its real entire majorant and minorant of exponential type $2\pi$.}
\end{figure}


\subsubsection*{Proof of part (i) of Lemma \ref{lem:extremalfunctions}}  From \eqref{rep f} and \eqref{rep M} we observe that 
\begin{align} \label{eqn:sumofGDelta1S_proof}
\begin{split}
M_{\Delta}^+(z)  = \sum_{\stackrel{n=-\infty}{n \ne 0}}^{\infty} &\left(\frac{\sin \pi \bigl(z-n \bigr)}{\pi\bigl(z-n \bigr)}\right)^2 
 \left\{f\left(\frac{n}{\Delta}\right) + \frac{\bigl(z-n \bigr)}{\Delta} f^{\prime}\left(\frac{n}{\Delta}\right)\right\} \\
 & \ \ \ \ \ \ \ \ \ \ \ \  - \sum_{\stackrel{n=-\infty}{n \neq 0}}^{\infty} \left(\frac{\sin \pi z}{\pi z}\right)^2  \frac{ z}{\Delta} f^{\prime}\left(\frac{n}{\Delta}\right)  + \frac{\pi}{2}  \left( \frac{\sin \pi z}{\pi z}\right)^2.
 \end{split}
\end{align}
Note that for any complex number $\xi$ we have $|\sin (\pi \xi)/ (\pi \xi)|^2 \ll e^{2 \pi |{\rm Im} \xi|}/ (1 + |\xi|^2)$. Next, we show that there is a constant $A>0$ such that
\begin{equation} \label{ineq:f}
|f(x)| \leq \frac{A}{1+ x^2},
\end{equation}
for all $x$. To prove (\ref{ineq:f}), we consider two cases: $|x| \leq 2$ and $|x| > 2.$ In the first case, $|f(x)|$ is bounded. In the second case, we write $f(x)$ as a power series
\begin{equation}\label{eqn:powerseriesofF}
f(x) = \sum_{n = 1}^{\infty} (-1)^{n - 1} \frac{2n}{(2n + 1)x^{2n + 1}}.
\end{equation}
Therefore (\ref{ineq:f})  follows from (\ref{eqn:powerseriesofF}) by the absolute convergence of this series. Moreover, notice that there is a constant $B >0$ such that for all real $x$ we have
\begin{equation} \label{ineq:fprime}
|f'(x)| = \left|\frac{-2}{(x^2+1)^2}\right|  \leq \frac{B}{|x|(x^2 + 1)}.
\end{equation}
From (\ref{eqn:sumofGDelta1S_proof}), (\ref{ineq:f}), and (\ref{ineq:fprime}), we can conclude that
\begin{equation*}
 \big|M_{\Delta}^+(x+iy)\big| \ll  \frac{\Delta^2  }{1+ |x+iy|}e^{2\pi |y|},
\end{equation*}
and hence from (\ref{def_g_m}) we derive (\ref{exp_bound_g}).

\smallskip

Now let $x$ be real. Since $f(x) = -f(-x)$ and $f'(x) = f'(-x)$ we can pair the terms $n$ and $-n$ in the sum 
\eqref{eqn:sumofGDelta1S_proof} to get
\begin{align}\label{rep_M1}
\begin{split}
M_{\Delta}^+(x)  & = \sum_{n=1}^{\infty} \left(\frac{\sin \pi \bigl(x-n \bigr)}{\pi\bigl(x^2-n^2 \bigr)}\right)^2 
 4 x n \, f\!\left(\frac{n}{\Delta}\right) + \sum_{n= 1}^{\infty}\frac{\sin^2 \pi \bigl(x-n \bigr)}{\pi^2\bigl(x^2-n^2 \bigr)} \frac{2x}{\Delta} f^{\prime}\!\left(\frac{n}{\Delta}\right)\\
 & \ \ \ \ \ \ \ \ \ \ \ \ \ \ - \sum_{n=1}^{\infty} \left(\frac{\sin \pi x}{\pi x}\right)^2  \frac{ 2x}{\Delta} f^{\prime}\!\left(\frac{n}{\Delta}\right) + \frac{\pi}{2}  \left( \frac{\sin \pi x}{\pi x}\right)^2\\
 & = \sum_{n=1}^{\infty} \left(\frac{\sin \pi \bigl(x-n \bigr)}{\pi\bigl(x^2-n^2 \bigr)}\right)^2 
 4 x n \, f\!\left(\frac{n}{\Delta}\right) + \sum_{n= 1}^{\infty}\frac{\sin^2 \pi \bigl(x-n \bigr)}{\pi^2\bigl(x^2-n^2 \bigr)x^2} \frac{2n^2x}{\Delta} f^{\prime}\!\left(\frac{n}{\Delta}\right)\\
 & \ \ \ \ \ \ \ \ \ \ \ \ \ \ + \frac{\pi}{2}  \left( \frac{\sin \pi x}{\pi x}\right)^2 \\
 &:= J_{\Delta}(x)  + \frac{\pi}{2}  \left( \frac{\sin \pi x}{\pi x}\right)^2.
 \end{split}
\end{align}
From \eqref{def_g_m}, we see that in order to prove \eqref{ineq:boundforgDeltax} it suffices to show that
\begin{equation}\label{bound_for_M_1}
|M_{\Delta}^+(x)| \ll \frac{\Delta^2}{\Delta^2 + x^2}.
\end{equation}
Since $\Delta\geq 1$, we observe that 
\begin{equation}\label{Sec3.10}
\frac{\pi}{2}  \left( \frac{\sin \pi x}{\pi x}\right)^2 \leq \frac{C}{1+x^2} \leq \frac{C \Delta^2}{\Delta^2 + x^2},
\end{equation}
for some constant $C$. Moreover, from \eqref{ineq:f} we also know that
\begin{equation}\label{Sec3.11}
|H_{\Delta}(x)| \leq \frac{A \Delta^2}{\Delta^2 + x^2}.
\end{equation}
Since 
\begin{equation*}
M_{\Delta}^+(x)  =  J_{\Delta}(x)  + \frac{\pi}{2}  \Big( \frac{\sin \pi x}{\pi x}\Big)^2 \geq H_{\Delta}(x)
\end{equation*}
for all real $x$, we conclude from \eqref{Sec3.10} and \eqref{Sec3.11} that there exists a constant $C$ such that 
\begin{equation*}
J_{\Delta}(x)  \geq - \frac{C \Delta^2}{\Delta^2 + x^2}.
\end{equation*}
On the other hand, for $x <0$, using that $f(\tfrac{n}{\Delta}) \geq 0$ and $f'(\tfrac{n}{\Delta}) \leq 0$ for $n\geq0$ in \eqref{rep_M1}, we obtain
\begin{equation*}
J_{\Delta}(x) \leq  \sum_{n= 1}^{|x|}\frac{\sin^2 \pi \bigl(x-n \bigr)}{\pi^2\bigl(x^2-n^2 \bigr)x} \frac{2n^2}{\Delta} f^{\prime}\!\left(\frac{n}{\Delta}\right) \leq  \frac{C \Delta^2}{\Delta^2 + x^2},
\end{equation*}
for some constant $C$. One can see this, for instance, by using \eqref{ineq:fprime} and splitting the sum in two parts: $n$ close to $x$ (say $n >x/2$) and $n$ small ($n < x/2$). Thus, we have proved that
\begin{equation}\label{ineqJ}
|J_{\Delta}(x) | \ll \frac{\Delta^2}{\Delta^2 + x^2}
\end{equation}
for $x <0$. Since $J_{\Delta}$ is an odd function, it follows that \eqref{ineqJ} holds for $x>0$ as well. This establishes \eqref{bound_for_M_1} and therefore completes the proof of part (i) of Lemma \ref{lem:extremalfunctions}.


\subsubsection*{Proof of parts (ii) and (iii) of Lemma \ref{lem:extremalfunctions}} From \eqref{def_g_m} and \eqref{rep_M1}, we see that 
\begin{equation*}
m_{\Delta}^+(x) = J_{\Delta}(\Delta x)  +  \frac{\pi}{2}  \left( \frac{\sin \pi \Delta x}{\pi \Delta x}\right)^2
\end{equation*}
where $J_{\Delta}$ is an odd function. Since $f(x)$ is odd and absolutely integrable, we have
\begin{equation*}
  \int_{-\infty}^{\infty} \big\{m_{\Delta}^+(x) - f(x)\big\}\, \dx =   \int_{-\infty}^{\infty} m_{\Delta}^+(x)\, \dx =  \int_{-\infty}^{\infty}   \frac{\pi}{2}  \left( \frac{\sin \pi \Delta x}{\pi \Delta x}\right)^2 \dx = \frac{\pi}{2\Delta}.
\end{equation*}
This proves part (iii) of Lemma \ref{lem:extremalfunctions}.
\smallskip 

The fact that the Fourier transform $\widehat{m}_{\Delta}^+$ is compactly supported on the interval $[-\Delta, \Delta]$ follows from the Paley-Wiener theorem since, by part (i), $m_{\Delta}^+$ has exponential type $2\pi \Delta$ and is integrable on $\R$. Moreover, using (\ref{ineq:boundforgDeltax}), we have
\begin{align*}
\big|\widehat{m}_{\Delta}^+(\xi)\big| &= \left| \int_{-\infty}^{\infty} m_{\Delta}^+(x) e^{-2\pi i x \xi }\, \dx\right| 
\leq \int_{-\infty}^{\infty} \left| m_{\Delta}^+(x) \right| \dx \ll \int_{-\infty}^{\infty} 
\frac{1}{1+x^2} \dx \ll 1.
\end{align*}
This completes the proof of part (ii), and thus establishes Lemma  \ref{lem:extremalfunctions}.


\section{Appendix A}
In this appendix, we derive the properties of the measure $\dmu_{\Delta}(\lambda)$ the were used in the proof of Lemma  \ref{lem:extremalfunctions}.
\subsection*{A1} First we prove that, for all $x>0$,  
\begin{align}
\begin{split} \label{App1}
 H_{\Delta}(x) &= \arctan \left(\frac{\Delta}{x}\right) - \frac{\Delta x}{x^2 + \Delta^2} \\
& =   \int_0^{\infty} \int_0^{\infty} e^{-\pi \lambda x^2} \, \frac{t}{2\sqrt{\pi \lambda^3}} \, e^{-\tfrac{t^2}{4\lambda}} \left( \frac{1}{t} \sin (\sqrt{\pi}\Delta t) - \Delta\sqrt{\pi} \cos (\sqrt{\pi} \Delta t)\right) \dt \,\dl.
\end{split}
\end{align}
By making the change of variables $y=\sqrt{\pi}\Delta t$ and using Fubini's theorem, we see that the right-hand side of \eqref{App1} is equal to 
\begin{equation}\label{App2}
\int_0^{\infty} \left\{\int_0^{\infty} \frac{e^{-\pi \lambda x^2 - \tfrac{y^2}{4 \pi \Delta^2 \lambda}}}{2 \pi \Delta \lambda ^{3/2}} \,y\, \dl \right\} \left(\frac{\sin y}{y} -   \cos y \right)\, \dy.
\end{equation}
Call $W(x,y,\Delta)$ the quantity inside the brackets in \eqref{App2}. To prove \eqref{App1}, it suffices to show that  $W(x,y,\Delta) = e^{\tfrac{-xy}{\Delta}}$. For this, consider the change of variables $k=\tfrac{\sqrt{2 \pi \Delta x \lambda}}{ \sqrt {y}}$ which implies that
\begin{equation*}
W(x,y,\Delta) = \frac{\sqrt{2xy}}{\sqrt{\pi \Delta}} \,e^{-\tfrac{xy}{\Delta}}\, \int_0^{\infty} \frac{e^{-\tfrac{xy}{2\Delta} \big(k - \tfrac{1}{k}\big)^2}}{k^2}\, \dk.
\end{equation*}
Now from the symmetry $k \to \tfrac{1}{k}$, we can re-write the last expression as
\begin{equation*}
W(x,y,\Delta) = \frac{1}{2} \frac{\sqrt{2xy}}{\sqrt{\pi \Delta}} \,e^{-\tfrac{xy}{\Delta}}\, \int_0^{\infty} e^{-\tfrac{xy}{2\Delta} \big(k - \tfrac{1}{k}\big)^2} \left(1 + \frac{1}{k^2}\right)\, \dk.
\end{equation*}
Finally, from the change of variables $w=k - \tfrac{1}{k}$, we arrive at 
\begin{equation*}
W(x,y,\Delta) = \frac{1}{2} \frac{\sqrt{2xy}}{\sqrt{\pi \Delta}} \,e^{-\tfrac{xy}{\Delta}}\, \int_{-\infty}^{\infty} e^{-\tfrac{xy}{2\Delta} w^2}\, \dw = e^{\tfrac{-xy}{\Delta}}.
\end{equation*}
This proves \eqref{App1}.

\subsection*{A2} We now prove that the measure $\dmu_{\Delta}(\lambda)$ given by \eqref{measure} is non-negative. We do so by establishing that the density function 
\begin{equation*}
D(\Delta, \lambda) = \int_0^{\infty} \frac{t}{2\sqrt{\pi \lambda^3}} \,e^{-\tfrac{t^2}{4\lambda}} \left( \frac{1}{t} \sin (\sqrt{\pi}\Delta t) - \Delta\sqrt{\pi} \cos (\sqrt{\pi} \Delta t)\right) \dt
\end{equation*}
is non-negative for all $\lambda > 0$ and all $\Delta >0$. Again, we make the variable change $y=\sqrt{\pi}\Delta t $ and obtain that
\begin{equation*}
D(\Delta, \lambda) = \frac{1}{2 \pi \Delta \lambda ^{3/2}}\int_0^{\infty} e^{- \tfrac{y^2}{4 \pi \Delta^2 \lambda}}\, (\sin y -   y \cos y )\, \dy.
\end{equation*}
Setting $ \pi \Delta^2 \lambda =  a^2$, it suffices to prove that 
\begin{equation*}
\int_0^{\infty} e^{- \tfrac{y^2}{4a^2}}\, (\sin y -   y \cos y )\, \dy \geq 0
\end{equation*}
for all $a >0$. Using integration by parts and the Fourier transform of the odd Gaussian, we obtain that
\begin{equation*}
\begin{split}
 \int_0^{\infty} e^{- \tfrac{y^2}{4a^2}}\, (\sin y -   y \cos y )\, \dy &= \left\{(1+2a^2) \int_0^{\infty} e^{- \tfrac{y^2}{4a^2}}\,\sin y \, \dy \right\} - 2a^2\\
& = \left\{(1+2a^2) \,2a\, e^{-a^2}\, \int_0^a e^{w^2}\,\dw \right\} - 2a^2.
\end{split}
\end{equation*}
We are left to prove that 
\begin{equation*}
h(a) = \int_0^a e^{w^2} \dw - \frac{a\, e^{a^2}}{1+2a^2} \geq 0
\end{equation*}
for all $a \geq 0$. This follows from observing that $h(0) = 0$ and 
\begin{equation*}
h'(a) = e^{a^2} \left( \frac{4a^2}{(1+2a^2)^2}\right) \geq 0
\end{equation*}
for all $a\geq0$. This concludes the proof.

\subsection*{A3} To verify that $\dmu_{\Delta}(\lambda)$ is indeed a finite measure on $(0,\infty)$, we note that \eqref{App1} and the monotone convergence theorem imply that
\begin{equation*}
\int_0^{\infty} \dmu_{\Delta}(\lambda) = \lim_{x \to 0^{+}} \int_0^{\infty} e^{-\pi \lambda x^2} \, \dmu_{\Delta}(\lambda)  = \lim_{x \to 0^{+}} \left\{\arctan \left(\frac{\Delta}{x}\right) - \frac{\Delta x}{x^2 + \Delta^2}\right\} = \frac{\pi}{2}.
\end{equation*}

\section*{Acknowledgments}
The authors would like to thank Kannan Soundararajan for his encouragement and for some valuable discussions, and Peter Sarnak for pointing out some of the references at early stages of the project. We would also like to thank J. Brian Conrey and Jeffrey D. Vaaler for helpful discussions on the topics presented here.

\bibliographystyle{amsplain}

\end{document}